\documentclass{article}
\usepackage[top=3cm, bottom=3cm, left=3cm, right=3cm]{geometry}
\usepackage[utf8]{inputenc}
\usepackage[T1]{fontenc}    
\usepackage{hyperref}       
\usepackage{url}            
\usepackage{booktabs}       
\usepackage{nicefrac}       
\usepackage{microtype}      
\usepackage{amsmath, amssymb, amsthm, amsfonts, fancyhdr, subcaption, xcolor, mdframed, array, tabu, multirow}
\usepackage{color}
\usepackage[export]{adjustbox}
\usepackage[bottom]{footmisc}

\newtheorem*{theorem*}{Theorem}

\theoremstyle{definition}

\newtheorem*{lemma*}{Lemma}
\usepackage{accents}
\newcommand\thickbar[1]{\accentset{\rule{2.0em}{.2pt}}{#1}}

\begin{document}

\begin{center}
    \LARGE A Short Proof of Köthe's Conjecture for Compact Rings
\end{center}

\begin{center}
     \large Scott Goodson\footnote{The University of Texas at Dallas, Scott.Goodson@utdallas.edu},  Alex Taylor\footnote{The University of Texas at Dallas, Alex.Taylor@utdallas.edu}
\end{center}

\begin{abstract}
\noindent We provide a new proof that the upper nilradical of a compact ring coincides with the sum of its left nil ideals using the properties of orthogonal idempotents in compact rings.\footnote{This paper was written as a part of the Directed Research Program at UT Dallas. The authors thank the mathematics department for providing space for us to work and Joseph Burnett for organizing the program.}
\end{abstract}

In this note all rings are assumed to be associative, but neither unital nor commutative. A topological ring is a pair $(R,\mathcal{T})$, where $R$ is a ring and $\mathcal{T}$ is a Hausdorff topology on $R$, such that $(R,+,\mathcal{T})$ is a topological group and multiplication $\cdot: R\times R\to R$ is continuous. We adopt the following standard notations and definitions. The sum of all left nil ideals of $R$ is denoted by $A(R)$. The upper nilradical of $R$ is the sum of all two-sided nil ideals of $R$ and is denoted by $N(R)$.  The Jacobson radical is denoted by $J(R)$. A nontrivial idempotent is an idempotent which is neither 0 nor 1, and a central idempotent is an idempotent belonging to the center of the ring.

Perhaps the most persistent open problem in non-commutative ring theory is whether or not $A(R) = N(R)$ for an arbitrary ring $R$, the so-called \emph{Köthe problem}. Ursul \cite{u1} proved in 1984 that they do indeed coincide for a compact topological ring $R$ by studying nilrings of bounded index. In this note we provide a new, straightforward proof that $A(R) = N(R)$ for any compact topological ring $R$.

\begin{lemma*}
Let $R$ be a compact, unital ring which is not local.\footnote{A compact unital ring is not local if and only if it has nontrivial idempotents.}  Then any non-local idempotent $e \in R$ can be decomposed into a sum of nontrivial orthogonal idempotents.
\end{lemma*}
\begin{proof}
First suppose that $R$ is unital. We show that $1 \in R$ can be expressed as a sum of nontrivial orthogonal idempotents. Since $R$ is not local, it possesses a maximal set of nontrivial orthogonal idempotents $\{e_{i}\}_{i \in I}$. Denote their sum by $s = \sum_{i \in I}e_{i}$, which is a nonzero idempotent. Consider the compact subring $(1-s)R(1-s)$. If $(1-s)R(1-s)$ is not quasi-regular, then it contains a local idempotent $e' = (1-s)r(1-s)$ \cite{u2}. For each $i \in I$ we have
\[ e'e_{i} = (1-s)r(1-s)e_{i} = (1-s)r(e_{i}-se_{i}) = 0 \]
and similarly $e_{i}e' = 0$. Therefore $\{e_{i}\}_{i \in I}$ can be extended by adjoining $e'$, and this contradicts the maximality of the set. Therefore $(1-s)R(1-s)$ is quasi-regular, and so
\[
(1-s)R(1-s) = J((1-s)R(1-s)) \subseteq J(R),
\]
hence $(1-s) \in J(R)$ implies that $s = 1$. Thus, 1 can be decomposed as a sum of nontrivial orthogonal idempotents. Now in the case that $R$ is not necessarily unital, let $e$ be any non-local idempotent of $R$. Then $e$ is the identity element of the compact subring $eRe$, hence $e$ is a sum of nontrivial orthogonal idempotents in $eRe \subseteq R$.

\end{proof}

\begin{theorem*}
Let $R$ be a compact topological ring. Then $A(R) = N(R)$.
\end{theorem*}

\begin{proof}
Since $N(R) \subseteq 
\thickbar{N(R)} \subseteq \thickbar{A(R)} \subseteq J(R)$, it suffices to show that $S = \thickbar{A(R)}/\thickbar{N(R)} = \{0\}$. This quotient is compact as the continuous image of a compact ring, and it is Hausdorff because $\thickbar{N(R)}$ is closed. Suppose that $S$ is nonzero. If $S$ is quasi-regular then, since the Jacobson radical of a compact ring is topologically nilpotent \cite{k1}, we have
\[S = J(S) = J\big(\thickbar{A(R)}/\thickbar{N(R)}\big)
 \subseteq J(R/J(R)) = \{0\}, \]
a contradiction. Thus, we assume that $S$ is not quasi-regular. Suppose also that $S$ is semisimple, otherwise we can pass to $S/J(S)$. Therefore $S$ is unital, and $S$ is local if and only if $S$ is a division ring. Since the result follows trivially in the case that $S$ is a division ring, suppose that $S$ is not local. Then $S$ possesses a collection of nontrivial central idempotents $\{e_{i}\}_{i \in I}$ \cite{w1}. Note that these central idempotents are not local: if $e_{i}$ is local for some $i \in I$, then $e_{i}S$ is a local ring, hence a division ring. Furthermore, $e_{i}S(1-e_{i}) = 0$ so $e_{i}s(1-e_{i}) = 0$ for some $s \in S$, and since $e_{i}s \in e_{i}S$ is a unit it follows that $e_{i} = 1$ is trivial, a contradiction. Choose one central idempotent $f \in \{e_{i}\}_{i \in I}$. By the Lemma we can express $f$ as a sum of nontrivial orthogonal idempotents,
\[
f = \sum_{j \in J} f_{j} 
\]
which constitute a maximal set of orthogonal idempotents in the compact semisimple ring $fS$, and can be chosen to be central since $fS$ is topologically isomorphic to a product of matrix rings over a finite field \cite{k1}. Furthermore, $fS$ possesses a nonzero local idempotent $e \in fS$ \cite{u2}. Evidently $ef \in S$ is also a local idempotent because $(ef)S(ef) = e(fSf)e$ is local. Now

\[
ef = \sum_{j \in J} ef_{j}
\]

where $(ef_{j})^2 = ef_{j}ef_{j} = e(f_{j})^2e = ef_{j}$ for each $j \in J$ and $ef_{j}ef_{k} = ef_{j}f_{k}e = 0$ for all $j\not= k$. Note also that $e$ does not annihilate every $f_{j}$ because if $ef_{j} = 0$ for each $j \in J$ then we can adjoin $e$ to the collection $\{f_{j}\}_{j \in J}$, contradicting the fact that $\{f_{j}\}_{j \in J}$ is a maximal set of orthogonal idempotents in $fS$. We have expressed $ef$ as a sum of orthogonal idempotents, and since they are nontrivial $J$ must possess at least two indices. Thus, we can write

\begin{align*}
ef &= \sum_{j \in J} ef_{j} \\
&= ef_{n} + \sum_{j \in J\setminus{n}} ef_{j} \\
&= ef_{n} + ef_{n}' \\
&= g + g'
\end{align*}

where $g' = ef_{n}' = \sum_{j \in J\setminus{n}} ef_{j}$ is a nonzero idempotent orthogonal to $g = ef_{n}$. Note that $g \not= g'$ otherwise $g^2 = g = 0$. Now $ge = eg = g$, so $g \in e(fS)e$. Similarly, $g'e = eg' = g'$ so $g' \in e(fS)e$. Therefore $g$ and $g'$ are two different nonzero idempotents in the local ring $e(fS)e$, a contradiction. It follows that $\thickbar{A(R)}/\thickbar{N(R)} = 0$ and $A(R) = N(R)$.
\end{proof}

\end{document}